\documentclass[11pt,a4paper]{amsart}
\usepackage[english]{babel}
\usepackage{hyperref}
\usepackage{german}
\usepackage{amsmath}
\usepackage{hyperref}
\usepackage{amssymb}
\usepackage{amsthm}
\usepackage{fontenc}
\usepackage{enumerate}
\usepackage{a4wide}

\theoremstyle{plain}
\newtheorem{satz}{Theorem}[section]
\newtheorem*{Satz}{Main Result}
\newtheorem{lem}[satz]{Lemma}

\newtheorem{prop}[satz]{Proposition}
\newtheorem{con}[satz]{Conjecture}

\theoremstyle{definition}
\newtheorem{defi}[satz]{Definition} 
\newtheorem{bsp}[satz]{Example}

\theoremstyle{remark}

\newtheorem{bem}[satz]{Remark} 
\newtheorem*{Bem}{Remark} 

\numberwithin{equation}{section}

\newcommand{\floor}[1]{\left\lfloor#1\right\rfloor}
\newcommand{\maxk}[1]{\left\{#1\right\}}
\newcommand{\erz}[1]{\langle#1\rangle}
\newcommand{\La}{\lambda}
\newcommand{\LA}{\Lambda}

\DeclareMathOperator{\reg}{reg}
\DeclareMathOperator{\codim}{codim}
\DeclareMathOperator{\hilb}{Hilb}
\DeclareMathOperator{\var}{var}
\DeclareMathOperator{\supp}{supp}
\DeclareMathOperator{\spann}{span}
\DeclareMathOperator{\red}{r}
\DeclareMathOperator{\boxx}{Box}
\DeclareMathOperator{\height}{ht}

\begin{document}

\title[C-M regularity of seminormal simplicial affine semigroup rings]{Castelnuovo-Mumford regularity of seminormal simplicial affine semigroup rings}
\author{Max Joachim Nitsche}
\address{Max\mbox{\;}Planck\mbox{\;}Institute\mbox{\;}for\mbox{\;}Mathematics\mbox{\;}in\mbox{\;}the\mbox{\;}Sciences,\mbox{\;}Inselstrasse\mbox{\;}22,\mbox{\;}04103\mbox{\;}Leipzig,\mbox{\;}Germany}
\email{nitsche@mis.mpg.de}
\thanks{}
\date{\today}


\begin{abstract}

We show that the Eisenbud-Goto conjecture holds for (homogeneous) seminormal simplicial affine semigroup rings. Moreover, we prove an upper bound for the Castelnuovo-Mumford regularity in terms of the dimension, which is similar as in the normal case. Finally, we compute explicitly the regularity of full Veronese rings.

\end{abstract}

\maketitle

\section{Introduction}

Let $K$ be a field, and let $R=K[x_1,\ldots,x_n]$ be a standard graded polynomial ring, that is, all variables $x_i$ have degree $1$. Let $M$ be a finitely generated graded $R$-module. By \mbox{$H^i_{{R}_+}(M)$} we denote the $i$-th local cohomology module of $M$ with respect to the homogeneous maximal ideal $R_+$ of $R$, and we set \mbox{$a(H^i_{R_+}(M)):=\max\maxk{r\mid H^i_{R_+}(M)_r\not=0}$} with the convention \mbox{$a(0)=-\infty$}. The \emph{Castelnuovo-Mumford regularity} (or \emph{regularity} for short) $\reg M$ of $M$ is defined by
$$
\reg M:=\max\maxk{i+a(H^i_{R_+}(M))\mid i\geq0}.
$$
The regularity $\reg M$ is an important invariant, for example, the $i$-th syzygy module of $M$ can be generated by elements of degree smaller or equal to $\reg M+i$, moreover, one can use the regularity of a homogeneous ideal to bound the degrees in certain minimal Gr\"obner bases; for more information we refer to the paper of Eisenbud and Goto \cite{SEG}, and to Bayer and Stillman \cite{SBSCDM}. So it is natural to ask for bounds for the regularity of a homogeneous ideal $I$ of $R$; note that $\reg I=\reg R/I+1$. Denote by $\codim R/I:=\dim_{K}[R/I]_{1}-\dim R/I$ the codimension of $R/I$ and by $\deg R/I$ its degree. An open conjecture is
\begin{con}[Eisenbud-Goto \cite{SEG}]\label{Segg}
If $K$ is algebraically closed and $I$ is a homogeneous prime ideal of $R$, then
\[
\reg R/I\leq\deg R/I-\codim R/I.
\]
\end{con}
By a result of Gruson, Lazarsfeld, and Peskine \cite{SGLP} Conjecture~\ref{Segg} holds if $\dim R/I=2$. The Cohen-Macaulay case was proven by Treger \cite{SEGCM}, and the Buchsbaum case by St\"uckrad and Vogel \cite{SEGBB}. Conjecture~\ref{Segg} also holds if $\deg R/I\leq \codim R/I+2$ by a result of Hoa, St\"uckrad, and Vogel \cite{SHSVC2}, and in characteristic zero for smooth surfaces by Lazarsfeld \cite{SLSMSD3} and for certain smooth threefolds by Ran \cite{SRLDG}. Moreover, Giaimo \cite{SDGEGCC} showed that the conjecture still holds for connected reduced curves.\\

\pagebreak

Since the Eisenbud-Goto conjecture is widely open, it would be nice to prove it for more cases; in the following we will consider homogeneous simplicial affine semigroup rings. A~semigroup is called \emph{affine} if it is finitely generated and isomorphic to a submonoid of $(\mathbb Z^m,+)$ for some $m\in\mathbb N^+$. Let $B$ be an affine semigroup. The \emph{affine semigroup ring} $K[B]$ associated to $B$ is defined as the $K$-vector space with basis $\{t^b\mid b\in B\}$ and multiplication given by the $K$-bilinear extension of $t^a\cdot t^b=t^{a+b}$. Since $B$ is an affine semigroup we have $G(B)\cong\mathbb Z^m$ for some $m\in\mathbb N$; where $G(B)$ denotes the group generated by $B$. Hence $G(B)\otimes_\mathbb Z\mathbb R$ is a finite dimensional $\mathbb R$-vector space with canonical embedding $G(B)\subseteq G(B)\otimes_\mathbb Z\mathbb R$ given by $x\mapsto x\otimes1$. We say that $B$ is \emph{simplicial} if the corresponding cone $C(B)$ is generated by linearly independent elements, where $C(X):=\{\sum_{i=1}^k r_ix_i\mid k\in\mathbb N^+, r_i\in\mathbb R_{\geq0}, x_i\in X\}$ for $X\subseteq G(B)\otimes_\mathbb Z\mathbb R$. An element $x\in B$ is called a \emph{unit} if $-x\in B$. We say that $B$ is \emph{positive} if $0$ is its only unit. In this case, the \emph{Hilbert basis} $\hilb(B)$, that is, the set of irreducible elements of $B$, is a unique minimal generating set of $B$; an element $x\in B$ is called \emph{irreducible} if it is not a unit and if for $x=y+z$ with $y,z\in B$ it follows that $y$ or $z$ is a unit. Moreover, we say that $B$ is \emph{homogeneous} if $B$ is positive and there is a positive $\mathbb Z$-grading on $K[B]$ in which every $t^b$ for $b\in\hilb(B)$ has degree $1$. See \cite[Chapter~2]{SBGPRKT}. In the following we will assume that $B$ is homogeneous. We will always consider the above $\mathbb Z$-grading on $K[B]$, moreover, by $\reg K[B]$ we mean the regularity of $K[B]$ with respect to the canonical $R$-module structure which is induced by the homogeneous surjective $K$-algebra homomorphism
$$
\pi: R=K[x_1,\ldots,x_{n}]\twoheadrightarrow K[B],
$$
given by $x_i\mapsto t^{a_i}$; where $\hilb(B)=\{a_1,\ldots,a_{n}\}$. Hence $R/\ker\pi\cong K[B]$, where $\ker\pi$ is a homogeneous prime ideal of $R$. In case that $B$ is simplicial we will also call $K[B]$ a simplicial affine semigroup ring.\\

By extending the ground field if necessary, (the inequality in) Conjecture~\ref{Segg} holds for $K[B]$ in particular if $\dim K[B]=2$, if $K[B]$ is Buchsbaum, and if \mbox{$\deg K[B]\leq \codim K[B]+2$}. The conjecture also holds if $\codim K[B]=2$ by Peeva and Sturmfels \cite{Scodim2} and for simplicial affine semigroup rings with isolated singularity by Herzog and Hibi \cite{SCMHH}. In \cite[Theorem~3.2]{SHSCM}, Hoa and St"uckrad presented a very good bound for the regularity of simplicial affine semigroup rings, moreover, they provided some cases where Conjecture~\ref{Segg} holds in the simplicial case. However, the Eisenbud-Goto conjecture is still widely open even for simplicial affine semigroup rings. In case that $B$ is simplicial and seminormal (see Definition~\ref{Sdefsemi}) we can confirm the Eisenbud-Goto conjecture for $K[B]$, we obtain the following

\begin{Satz}[Theorem~\ref{Sd-1}, Theorem~\ref{Segsn}]\label{Smainresult}

Let $K$ be an arbitrary field and let $B$ be a homogeneous affine semigroup. If $B$ is simplicial and seminormal, then
$$
\reg K[B]\leq \min\maxk{\dim K[B]-1, \deg K[B]-\codim K[B]}.
$$

\end{Satz}

This result is more or less well known if $B$ is normal (see Definition~\ref{Sdefsemi}), since $K[B]$ is Cohen-Macaulay in this case; see \cite[Theorem~1]{SMHCM}, \cite[Theorem~6.10]{SBGPRKT}, and Remark~\ref{Snormal}. In fact, the ring $K[B]$ is not necessary Buchsbaum if $B$ is simplicial and seminormal, see Example~\ref{Sdreimehr}. To prove Conjecture~\ref{Segg} in the seminormal simplicial case we will~use~an idea of Hoa and St\"uckrad and decompose the ring $K[B]$ into a direct sum of certain monomial ideals. This becomes even more powerful in this case, since seminormality of simplicial affine semigroup rings can be characterized in terms of the decomposition by a result of Li \cite{SPHDPL}.\\

In Section~\ref{Sbasics} we will recall the decomposition of simplicial affine semigroup rings. Moreover, we will introduce sequences with $*$-property which will be useful to prove the main result in Section~\ref{Sseminormalcase}. Finally, we will compute explicitly the Castelnuovo-Mumford regularity of full Veronese rings in Section~\ref{Sfullveronese}. We set $M_{d,\alpha}:=\{(u_{1},\ldots,u_{d})\in\mathbb N^d \mid \sum_{i=1}^{d}u_{i}=\alpha\}$ where $d,\alpha\in\mathbb N^+$, moreover, we define $B_{d,\alpha}$ to be the submonoid of $(\mathbb N^d,+)$ which is generated by $M_{d,\alpha}$. In Theorem~\ref{Sberechnung} we will show that \mbox{$\reg K[B_{d,\alpha}] = \floor{d - \frac{d}{\alpha}}$}. For a general consideration of seminormal rings we refer to \cite{SSOS, STSPG}, and for unspecified notation to \cite{SBGPRKT, SEB}.

\section{Basics}\label{Sbasics}

In the following we will assume that the homogeneous affine semigroup $B$ is simplicial, that is, there are linearly independent elements $e_1,\ldots,e_d\in C(B)$ such that $C(B)=C(\{e_1,\ldots,e_d\})$. Without loss of generality we may assume that ${e_1,\ldots,e_d}\in \hilb(B)$. Consider the $\mathbb R$-vector space isomorphism \mbox{$\varphi:\spann(\{e_1,\ldots,e_d\})\rightarrow\mathbb R^d$} where $e_i$ is mapped to the element in $\mathbb N^d$ all of whose coordinates are zero except the $i$-th coordinate which is equal to $\alpha$ for some $\alpha\in\mathbb N^+$, that is, $\varphi(e_i)=(0,\ldots,0,\alpha,0,\ldots,0)$. By construction we have $\varphi(B)\subseteq \mathbb R^d_{\geq0}$, since $C(B)=C(\{e_1,\ldots,e_d\})$, hence $\varphi(B)\subseteq \mathbb Q^d_{\geq0}$ by the Gaussian elimination. Thus, by choosing a suitable $\alpha$ we may assume that $\varphi(\hilb(B))\subset \mathbb N^d$, or equivalently, $\varphi(B)\subseteq \mathbb N^d$. The affine semigroup $\varphi(B)$ is again homogeneous, it follows that the coordinate sum of all elements of $\varphi(\hilb(B))$ is equal to $\alpha$, see \cite[Proposition~2.20]{SBGPRKT}. Note that we can compute $\reg K[B]$ in terms of $H^i_{K[{B}]_+}(K[B])$, since $H^i_{R_+}(K[B])\cong H^i_{K[{B}]_+}(K[B])$, see \cite[Theorem~13.1.6]{SBSLC}; where $K[B]_+$ denotes the homogeneous maximal ideal of $K[B]$. The isomorphism $B\cong\varphi(B)$ of semigroups induces an isomorphism of $\mathbb Z$-graded rings $K[B]\cong K[\varphi(B)]$. This enables us to identify a homogeneous simplicial affine semigroup $B$ with its image $\varphi(B)$ in $\mathbb N^d$. Thus, we may assume that $B$ is the submonoid of $(\mathbb N^d,+)$ which is generated by a set $\{e_1,\ldots,e_d,a_1,\ldots,a_c\}\subseteq M_{d,\alpha}$, where
$$
e_1:=(\alpha,0,\ldots,0), e_2:=(0,\alpha,0,\ldots,0),\ldots,e_d:=(0,\ldots,0,\alpha).
$$
Let $a_i=(a_{i[1]},\ldots,a_{i[d]})$; since $\alpha\in\mathbb N^+$ can be chosen to be minimal, we may assume that the integers $a_{i[j]},\:i=1,\ldots,c,\:j=1,\ldots,d$, are relatively prime. Moreover, we assume that $c\geq1$, since the case $c=0$ is not relevant in our context. Note that $K$ is an arbitrary field, $\dim K[B]=d$, and $\codim K[B]=c$. Our notation tries to follow the notation in \cite{SHSCM}.\\

By $x_{[i]}$ we denote the $i$-th component of $x$ and $\deg x:=(\sum_{j=1}^dx_{[j]})/\alpha$, for $x\in G(B)$. We define $A:=\erz{e_1,\ldots,e_d}$ to be the submonoid of $B$ generated by $e_1,\ldots,e_d$, and we set
$$
B_A:=\{x\in B\mid x-a\notin B~\forall a\in A\setminus\{0\}\}.
$$ 
Note that $B_A$ is finite. Moreover, if $x\notin B_A$ then $x+y\notin B_A$ for all $x,y\in B$. We define $x\sim y$ if $x - y\in G(A)=\alpha\mathbb Z^d$, thus, $\sim$ is an equivalence relation on $G(B)$. Every element of $G(B)$ is equivalent to an element of $G(B)\cap D$, where $D:=\{(x_{[1]},\ldots,x_{[d]})\in\mathbb Q^d\mid 0\leq x_{[i]}<\alpha~\forall i\}$ and for all $x,y\in G(B)\cap D$ with $x\not=y$ we have $x\not\sim y$. Hence the number of equivalence classes $f:=\#(G(B)\cap D)$ on $G(B)$ is finite, moreover, there are also $f$ equivalence classes on $B$ and on $B_{A}$. By $\Gamma_1,\ldots,\Gamma_f$ we denote the equivalence classes on $B_{A}$. For $t=1,\ldots,f$ we define
$$
h_t:=(\min\maxk{m_{[1]}\mid m\in \Gamma_t}, \min\maxk{m_{[2]}\mid m\in \Gamma_t},\ldots,\min\maxk{m_{[d]}\mid m\in \Gamma_t}).
$$
Note that $x-h_t\in A$ for all $x\in\Gamma_t$. This shows that $h_t\in G(B)\cap \mathbb N^d$. Let $T:=K[y_1,\ldots,y_d]$ be a standard graded polynomial ring, that is, all variables $y_i$ have degree $1$. We define $\tilde \Gamma_t :=\{y^{(x-h_t)/\alpha}\mid x\in\Gamma_t\}$, where $u/\alpha:=(u_{[1]}/\alpha,\ldots,u_{[d]}/\alpha)$ and $y^{u}:=y_1^{u_{[1]}}\cdot\ldots\cdot y_d^{u_{[d]}}$ for $u=(u_{[1]},\ldots,u_{[d]})\in\mathbb N^d$. We obtain $\tilde\Gamma_t\subset T$, and therefore $I_t:= \tilde\Gamma_t T$ is a monomial ideal in $T$ for all $t=1,\ldots,f$. It follows that $\height I_t\geq2$ (height), since $\gcd \tilde\Gamma_t=1$. By $T_+$ we denote the homogeneous maximal ideal of $T$. See \cite[Section~2]{SHSCM}. We have:

\begin{prop}[{\cite[Proposition~2.2]{SHSCM}}]\label{Szerlmi}

There are isomorphisms of $\mathbb Z$-graded $T$-modules:

\begin{enumerate}

\item $K[B] \cong \bigoplus_{t=1}^f I_t (-\deg h_t).$

\item $H^i_{K[{B}]_+}(K[B]) \cong \bigoplus_{t=1}^f H^i_{T_+} (I_t) (-\deg h_t)$ for all $i\geq0$.

\end{enumerate}

\end{prop}

It follows that  $\deg K[B]=f$. Moreover, we have
\begin{equation}\label{Sregberechnung}
\reg K[B] = \max\maxk{\reg I_t +\deg h_t\mid t=1,\ldots,f},
\end{equation}
where $\reg I_t$ denotes the regularity of $I_t$ as a $\mathbb Z$-graded $T$-module. This shows that the regularity of $K[B]$ is independent of $K$ for $\dim K[B]\leq5$ by \cite[Corollary~1.4]{SUBZ}.

\begin{Bem}

This decomposition can be computed by using the \textsc{Macaulay2} \cite{SM2} package \textsc{MonomialAlgebras} \cite{SBEN}, which has been developed by Janko B\"ohm, David Eisenbud, and the author. In this package we consider the case of affine semigroups $Q'\subseteq Q\subseteq\mathbb N^d$ such that $K[Q]$ is finite over $K[Q']$; the implemented algorithm decomposes the ring $K[Q]$ into a direct sum of monomial ideals in $K[Q']$. There is also an algorithm implemented computing $\reg K[Q]$ in the homogeneous case, moreover, there are functions available testing the Buchsbaum, Cohen-Macaulay, Gorenstein, normal, and the seminormal property in the simplicial case. Note that this decomposition works more general, for more information we refer to \cite{SBENAX}.

\end{Bem}

\begin{defi}

For an element $x\in B$ we say that a sequence $\lambda=(b_1,\ldots,b_n)$ has \emph{$*$-property} if $b_1,\ldots,b_n \in \{e_1,\ldots,e_d,a_1,\ldots,a_c\}$ and $x-b_1\in B, x-b_1-b_2\in B,\ldots,x-(\sum_{j=1}^n b_j)\in B$; we say that the \emph{length} of $\La$ is $n$. Let $\lambda=(b_1,\ldots,b_n)$ be a sequence with $*$-property of $x$; we define $x(\lambda,i):=x-(\sum_{j=1}^i b_j)$ for $i=1,\ldots,n$, and $x(\lambda,0):=x$. By $\Lambda_x$ we denote the set of all sequences with $*$-property of $x$ with length $\deg x$, with the convention $\LA_0:=\emptyset$.

\end{defi}

By construction we have $\LA_x\not=\emptyset$ for all $x\in B\setminus\{0\}$. The definition of a sequence with \mbox{$*$-property} is motivated to control the degree of $K[B]$, the second assertion in Lemma~\ref{S*prop} illustrates the usefulness of this construction. For elements $x,y\in G(B)$ we define $x\geq y$ if $x_{[k]}\geq y_{[k]}$ for all $k=1,\ldots,d$.

\begin{bem}\label{Sxdegx}

Let $\lambda=(b_1,\ldots,b_n)$ be a sequence with $*$-property of $x$. We get $x(\lambda,i)\geq x(\lambda,j)$ for $0\leq i\leq j\leq n$. Moreover, we have $\deg x(\lambda,i)=\deg x-i$ for $i=0,\ldots,n$. Hence for $\La\in\LA_x$ we get $x(\La,\deg x)=0$.

\end{bem}

\begin{lem}\label{S*prop}

Let $x\in B_A\setminus\{0\}$ and $\lambda=(b_1,\ldots,b_n)$ be a sequence with $*$-property of $x$. Then

\begin{enumerate}

\item $x(\lambda,i)\in B_A$ for all $i=0,\ldots,n$.

\item $x(\lambda,i)\not\sim x(\lambda,j)$ for all $i,j\in\mathbb N$ with $0\leq i<j\leq n$.

\end{enumerate}

\end{lem}
\begin{proof}

$(1)$ Follows from construction since if $y\notin B_A$ then $y+z\notin B_A$ for all $y,z\in B$.\\
$(2)$ Suppose to the contrary that $x(\lambda,i)\sim x(\lambda,j)$ for some $i,j\in\mathbb N$ with $0\leq i<j\leq n$. We have $x(\lambda,i)\geq x(\lambda,j)$, hence 
$$
x(\lambda,i)=x(\La,j)+\sum\nolimits_{t=1}^d n_te_t
$$
for some $n_t\in\mathbb N$. Since $\deg x(\lambda,i)>\deg x(\lambda,j)$ we get that $n_t>0$ for some $t\in\{1,\ldots,d\}$. Thus, $x(\La,i)-e_t\in B$ and therefore $x(\lambda,i)\notin B_A$ which contradicts claim $(1)$.
\end{proof}

\begin{bem}\label{Sexistenz}

Let $x\in B_A\setminus\{0\}$ and $\La=(b_1,\ldots,b_n)$ be a sequence with $*$-property of $x$. Suppose that $b_j\in\{e_1,\ldots,e_d\}$ for some $j\in\{1,\ldots,n\}$. Hence $x-b_j=x(\La,n)+\sum_{k=1, k\not=j}^nb_k\in B$ which contradicts $x\in B_A$. This shows that $b_1,\ldots,b_n\in\{a_1,\ldots,a_c\}$.

\end{bem}

Lemma~\ref{S*prop} implies that $\deg x\leq \deg K[B]-1$ for all $x\in B_A$. This bound can be improved by using the following observation:

\begin{bem}\label{Sa1bisc}

Consider the set \mbox{$L=\{0,a_1,\ldots,a_c\}$}, by construction $L\subseteq B_A$. Let $x\in L$ and $y\in B_A$ with $x\not= y$; suppose that $x\sim y$. Since $0\leq x_{[i]}<\alpha$ for all $i=1,\ldots,d$, we have $y\geq x$. By a similar argument as in Lemma~\ref{S*prop} $(2)$ we get $y\not\in B_A$. This shows that $x\not\sim y$. 

\end{bem}

\begin{prop}[{\cite[Theorem~1.1]{SHSCM}}]\label{Segred}

We have $\deg x\leq \deg K[B] - \codim K[B]$ for all $x\in B_A$.

\end{prop}
\begin{proof}

Let $x\in B_A\setminus\{0\}$, and $\lambda\in\LA_x$. By Lemma~\ref{S*prop} and Remark~\ref{Sa1bisc} we get a set 
$$
L:=\{0,a_1,\ldots,a_c\}\cup\{x(\lambda,0),\ldots,x(\La,\deg x-1)\},
$$ 
with $L\subseteq B_A$ such that $y\not\sim z$ for all $y,z\in L$ with $y\not=z$. Hence 
$$
\deg K[B]=f\geq \#L=\deg x + \codim K[B].
$$
\end{proof}

We note that this proof is a new proof of \cite[Theorem~1.1]{SHSCM}. We define the \emph{reduction number} $\red(K[B])$ of $K[B]$ by $\red(K[B]):=\max\maxk{\deg x\mid x\in B_A}$, see \cite[Page~129,135]{SHSCM}. It follows that
\begin{equation}
\red(K[B])\leq \deg K[B] - \codim K[B],
\end{equation}
that is, the Eisenbud-Goto conjecture holds for the reduction number of $K[B]$. So whenever we have $\reg K[B]=\red(K[B])$ the Eisenbud-Goto conjecture holds. It should be mentioned that this property does not hold in general. Even for a monomial curve in $\mathbb P^3$ the equality does not hold. For $B=\erz{(40,0), (0,40), (35,5), (11,29)}$ we get $\reg K[B]=13>11=\red (K[B])$. Note that we always have $\red (K[B])\leq \reg K[B]$ by Equation~(\ref{Sregberechnung}).

\begin{bsp}\label{Sbsprechnung}

Consider the monoid $B=\erz{(4,0), (0,4), (3,1), (1,3)}$. We have 
$$
B_A=\{(0,0), (3,1), (1,3), (6,2),(2,6)\},
$$ 
and therefore $\red (K[B])=\max\maxk{0,1,1,2,2}=2$. We get 
$$
\Gamma_1=\{(0,0)\}, \Gamma_2=\{(3,1)\}, \Gamma_3=\{(1,3)\}, \Gamma_4=\{(6,2), (2,6)\},
$$
and $h_1=(0,0),h_2=(3,1), h_3=(1,3),h_4=(2,2)$. By this we have $I_1=I_2=I_3=T$ and $I_4= (y_1,y_2)T$, hence 
$$
\reg K[B] = \max\maxk{\reg T + 0, \reg T +1, \reg T +1, \reg (y_1,y_2)T+1} = \max\maxk{0,1,1,2} = 2.
$$

\end{bsp}

\begin{lem}\label{Shilfe}

Let $x\in B_A, t\in\mathbb N^+, q\in\{1,\ldots,d\}$, and $x_{[q]}=t\alpha$. There exists a $\La\in\LA_x$ such that $(t-1)\alpha<x(\La,1)_{[q]}<t\alpha$. 

\end{lem}
\begin{proof}

Fix a $\nu=(b_1,\ldots,b_{\deg x})\in\Lambda_x$. We have $x(\nu,\deg x)=0$ by Remark~\ref{Sxdegx}, hence there is a $k\in\{1,\ldots,\deg x\}$ with $b_{k[q]}>0$. Since $b_k\in\{a_1,\ldots,a_c\}$ by Remark~\ref{Sexistenz} we get that $b_{k[q]}<\alpha$. The claim follows from the fact that $(b_{\sigma(1)},\ldots,b_{\sigma(\deg x)})\in\LA_x$ for every permutation $\sigma$ of $\{1,\ldots,\deg x\}$, since $x=\sum_{j=1}^{\deg x}b_j$.
\end{proof}

The next combinatorial Lemma will be useful to prove the Eisenbud-Goto conjecture in the seminormal case in Theorem~\ref{Segsn}.

\begin{lem}\label{Shilfe2}

Let $J\subseteq\{1,\ldots,d\}$ with $\#J\geq1$, and let $x\in B_A$ such that $x_{[q]}=\alpha$ for all $q\in J$. There exists a $\La\in\LA_x$ with the property: for all $p=1,\ldots,\#J$ there is a $q\in J$ such that $0<x(\La,p)_{[q]}<\alpha$.

\end{lem}
\begin{proof}

Using induction on $k\in\mathbb N^+$ with $k\leq\#J$ as well as Lemma~\ref{Shilfe} we get a sequence $\La=(b_1,\ldots,b_{\deg x})\in\LA_x$ with the property: for all $p=1,\ldots,k$ there is a $q\in J$ such that $0<x(\La,p)_{[q]}<\alpha$. In case that $x(\La,k)_{[q]}=\alpha$ for some $q\in J$ we can use Lemma~\ref{Shilfe} to get a sequence $(g_1,\ldots,g_{\deg x(\La,k)})\in\LA_{x(\La,k)}$ with $0<(x(\La,k)-g_1)_{[q]}<\alpha$, since $x(\La,k)\in B_A$ by Lemma~\ref{S*prop}. By construction it follows that
$$
\La'=(b_1,\ldots,b_k,g_1,\ldots g_{\deg x(\La,k)})\in\LA_{x},
$$
with the property: for all $p=1,\ldots,k+1$ there is a $q\in J$ such that $0<x(\La',p)_{[q]}<\alpha$. Assume that $x(\La,k)_{[q]}<\alpha$ for all $q\in J$. In this case $\La$ has already the claimed property. Fix a $p\in\{1,\ldots,\#J\}$; we need to show that there is a $q\in J$ with $x(\La,p)_{[q]}>0$ and we are done. Suppose to the contrary that $x(\La,p)_{[q]}=0$ for all $q\in J$. Since $\deg x(\La,p)\leq\deg x-\#J$ we get $p=\#J$ by Remark~\ref{Sxdegx}. Again by Remark~\ref{Sxdegx} it follows that  $x(\La,\#J)=x-(\sum_{q\in J}e_q)$ which contradicts  $x\in B_A$, since $x(\La,\#J)\in B$.
\end{proof}

\section{The seminormal case}\label{Sseminormalcase}

There are two closely related definitions:

\begin{defi}\label{Sdefsemi} 

Let $U$ be an affine semigroup. 

\begin{enumerate}

\item We call $U$ \emph{normal} if $x\in G(U)$ and $tx\in U$ for some $t\in\mathbb N^+$ implies that $x\in U$. 

\item We call $U$ \emph{seminormal} if $x \in G(U)$ and $2x, 3x \in U$ implies that $x \in U$.

\end{enumerate}

\end{defi}

A domain $S$ is called \emph{seminormal} if for every element $x$ in the quotient field $Q(S)$ of $S$ such that $x^2,x^3\in S$ it follows that $x\in S$. Note that the ring $K[U]$ is seminormal if and only if $U$ is seminormal. This was first observed by Hochster and Roberts in \cite[Proposition~5.32]{SHRLC}, provided that $U\subseteq \mathbb N^d$. For a proof in the general affine semigroup case we refer to \cite[Theorem~4.76]{SBGPRKT}. A similar result holds in the normal case, see \cite[Proposition~1]{SMHCM} and \cite[Theorem~4.40]{SBGPRKT}. To get new bounds for the regularity of $K[B]$, we need another characterization. We define the set $\boxx(B) := \{x\in B\mid x_{[i]}\leq \alpha~\forall i=1,\ldots,d\}$.

\begin{satz}[{\cite[Theorem~4.1.1]{SPHDPL}}]\label{Scharsn}

The simplicial affine semigroup $B$ is seminormal if and only if $B_A$ is contained in $\boxx(B)$.

\end{satz}

In case that $\Gamma_t\subseteq \boxx(B)$ for some $t\in\{1,\ldots,f\}$ we get $((x-h_t)/\alpha)_{[i]}\in\{0,1\}$ for all $x\in\Gamma_t$ and for all $i=1,\ldots,d$. Thus, $I_t$ is a squarefree monomial ideal in $T$ if $\Gamma_t\subseteq\boxx(B)$. This shows that all ideals in the decomposition are squarefree in the seminormal case.

\begin{lem}\label{Stechnisch}

Let $\Gamma_t\subseteq\boxx(B)$ for some $t\in\{1,\ldots,f\}$ with $\Gamma_t\not=\{0\}$. Let $x,y\in\Gamma_t$, and let $i\in\mathbb N$ with $1\leq i\leq d$. We have

\begin{enumerate}

\item If $x_{[i]}\not=y_{[i]}$, then $x_{[i]}-y_{[i]}\in\{-\alpha,\alpha\}$.

\item If $0<x_{[i]}<\alpha$, then $x_{[i]}=y_{[i]}$.

\item If $x_{[i]}\not=y_{[i]}$, then $x_{[i]}\in\{0,\alpha\}$ and $y_{[i]}=\alpha-x_{[i]}$.

\item \mbox{We have $0<x_{[j]}<\alpha$ and $0<x_{[k]}<\alpha$ for some $j,k\in\{1,\ldots,d\}$ with $j\not=k$.}

\item If $h_{t[i]}>0$, then $h_{t[i]}=x_{[i]}$.

\end{enumerate}

\end{lem}
\begin{proof}
$(1)$ We have $x_{[i]}-y_{[i]}\in\alpha\mathbb Z$ and $x_{[i]}-y_{[i]}\in[-\alpha,\alpha]$, since $0\leq x_{[i]},y_{[i]}\leq\alpha$. Hence $x_{[i]}-y_{[i]}\in\{-\alpha,\alpha\}$.\\
$(2)$ We have $x_{[i]}-y_{[i]}\notin\{-\alpha,\alpha\}$ and therefore $x_{[i]}=y_{[i]}$ by claim $(1)$.\\
$(3)$ By claim $(1)$ and $(2)$ we have $x_{[i]}-y_{[i]}\in\{-\alpha,\alpha\}$ and $x_{[i]}\in\{0,\alpha\}$. Hence $y_{[i]}=\alpha-x_{[i]}$.\\
$(4)$ Suppose that $0<x_{[j]}<\alpha$ for exactly one $j\in\{1,\ldots,d\}$, that is, $x_{[l]}\in\{0,\alpha\}$ for all $l\in\{1,\ldots,d\}\setminus\{j\}$. Hence $\sum_{l=1}^d x_{[l]}\notin\alpha\mathbb N$ which contradicts $x\in B$. If $x_{[l]}\in\{0,\alpha\}$ for all $l=1,\ldots,d$ we have $x\sim 0$. Hence $0\in\Gamma_t$, that is, $\Gamma_t=\{0\}$ which contradicts our assumption.\\
$(5)$ We have $0<h_{t[i]}\leq x_{[i]}\leq\alpha$ and therefore $h_{t[i]}=x_{[i]}$, since $h_{t[i]}-x_{[i]}\in\alpha\mathbb Z$.
\end{proof}

\begin{prop}[{\cite[Theorem~2.2]{SPRCMS}}]\label{Sbisdrei}

Let $B$ be seminormal. If $\dim K[B]\leq3$, then the ring $K[B]$ is Cohen-Macaulay.

\end{prop}
\begin{proof}

By \cite[Theorem~6.4]{SRSHFGA} we need to show that $\#\Gamma_t=1$ for all $t=1,\ldots,f$. We have $B_A\subseteq \boxx(B)$ by Theorem~\ref{Scharsn}. The case $d=2$ follows from Lemma~\ref{Stechnisch} $(4)$ and $(2)$. Let $d=3$; suppose to the contrary that $\#\Gamma_t\geq 2$ for some $t\in\{1,\ldots,f\}$. Let $x,y\in\Gamma_t$ with $x\not=y$. We get $0<x_{[j]}=y_{[j]}<\alpha$ and $0<x_{[k]}=y_{[k]}<\alpha$ for some $j,k\in\{1,2,3\}$ with $j\not=k$ by Lemma~\ref{Stechnisch} $(4)$ and $(2)$. By Lemma~\ref{Stechnisch} $(3)$ we may assume that $x_{[l]}=\alpha$ and $y_{[l]}=0$ for $l\in\{1,2,3\}\setminus\{j,k\}$. Hence $x-e_l=y\in B$ which contradicts $x\in B_A$.
\end{proof}

\begin{bsp}\label{Sdreimehr}

Proposition~\ref{Sbisdrei} does not hold for $\dim K[B]>3$. Consider the monoid
$$
B=\erz{e_1,\ldots,e_4,(1,1,0,0), (1,0,1,0), (0,0,1,1), (0,1,0,1)}\subset\mathbb N^4,
$$
with $\alpha=2$. We have $B_A\subseteq \boxx(B)$, thus, $B$ is seminormal by Theorem~\ref{Scharsn}. One can show that $(0,1,1,0)+e_1,(0,1,1,0)+e_4\in B$, but $(0,1,1,0)+e_3=(0,1,3,0)\notin B$. Hence $K[B]$ is not Buchsbaum by \cite[Lemma~3]{STBB}. Let $U$ be a seminormal positive affine semigroup. Note that $K[U]$ is Cohen-Macaulay if $K[U]$ is Buchsbaum by \cite[Proposition~4.15]{SBLRSN}.

\end{bsp}

\begin{bem}\label{Sredd-1}

Consider an element $x\in\boxx(B)\cap B_A$. Since $x_{[i]}\leq\alpha$ for all $i=1,\ldots,d$ we have $\deg x\leq d$. On the other hand there is only one element in $\boxx(B)$ with degree $d$, that is, $(\alpha,\ldots,\alpha)$, but $(\alpha,\ldots,\alpha)\notin B_A$. This shows that $\deg x\leq d-1$. By Theorem~\ref{Scharsn} we get $\red(K[B])\leq d-1$ if $B$ is seminormal. In Theorem~\ref{Sd-1} we obtain a similar bound for the regularity of $K[B]$ in the seminormal case.

\end{bem}

\begin{defi}

For a monomial $m=y_1^{c_1}\cdot\ldots\cdot y_d^{c_d}$ in $T$ we define $\deg m=\sum_{j=1}^dc_j$. Let $I$ be a monomial ideal in $T$ with minimal set of monomial generators $\{m_1,\ldots,m_s\}$. Let $F=y_1^{b_1}\cdot\ldots\cdot y_d^{b_d}$ be the least common multiple of $\{m_1,\ldots,m_s\}$. We define $\var(I):=\deg F$, moreover, we define the set $\supp (I)\subseteq\{1,\ldots,d\}$ by $i\in \supp (I)$ if $b_i\not=0.$

\end{defi}

\begin{bem}\label{Sgeq2}

Let $t\in\{1,\ldots,f\}$; we note that $\tilde\Gamma_t$ is always a minimal set of monomial generators of $I_t$. Moreover, every monomial ideal in $T$ has a unique minimal set of monomial generators. By construction we get that $I_t$ is a proper ideal in $T$ if and only if $\#\Gamma_t\geq2$. Since  $\height I_t\geq2$ we have $\var (I_t)\not=1$. Hence $I_t$ is a proper ideal if and only if $\var (I_t)\geq2$. Moreover, if $I_t$ is a proper ideal, then $\deg h_t\geq1$, since $h_t\in G(B)\cap\mathbb N^d$ and $h_t\not=0$.

\end{bem}

Consider the squarefree monomial ideal $I=(y_1y_2,y_2y_5y_6)T$ in $T=K[y_1,\ldots,y_6]$. We have $\var(I)=4$ and $\supp(I)=\{1,2,5,6\}$. So $\supp(I)$ is the set of indices of the variables which occur in the minimal generators of a monomial ideal $I$ in $T$. Note that we always have $\var(I)=\#\supp(I)$ in case that $I$ is a squarefree monomial ideal. Hence $\var(I_t)=\#\supp(I_t)$ if $\Gamma_t\subseteq\boxx(B)$ for some $t\in\{1,\ldots,f\}$.

\pagebreak

\begin{lem}\label{Svari}

Let $\Gamma_t\subseteq\boxx(B)$ for some $t\in\{1,\ldots,f\}$. Then $\var(I_t)\leq d-1-\deg h_t$.

\end{lem}
\begin{proof}

Let $\#\Gamma_t=1$; we get $\var(I_t)=0$ and $\deg h_t\leq d-1$ by Remark~\ref{Sredd-1}. So we may assume that $\#\Gamma_t\geq2$. Let $x\in\Gamma_t$; by Lemma~\ref{Stechnisch} $(4)$ there are some $j,k\in\{1,\ldots,d\}$ with $j\not=k$ such that $0<x_{[j]},x_{[k]}<\alpha$. Hence $0<h_{t[j]},h_{t[k]}<\alpha$, since $x-h_t\in A$. By Lemma~\ref{Stechnisch} $(5)$ we get that $h_{t[q]}=0$ for all $q\in \supp (I_t)$. We have $\#\supp (I_t)=\var(I_t)$, since $I_t$ is squarefree. Let $J:=\{1,\ldots,d\}\setminus \supp(I_t)$; we get $j,k\in J$ and $h_{t[q]}\leq\alpha$ for all $q\in J$ it follows that
$$
\deg h_t = \frac{1}{\alpha}\sum\nolimits_{q\in J} h_{t[q]}< d - \#\supp(I_t)=d-\var(I_t).
$$
\end{proof}

\begin{bem}\label{Snormal}

Consider a normal homogeneous affine semigroup $U$. One can show that $\reg K[U]\leq \dim K[U]-1$. This can be deduced from the proof of \cite[Corollary~4.7]{SCMGM} and \cite[Corollary~3.8]{SCMGM}, and the fact that $K[U]$ is Cohen-Macaulay by \cite[Theorem~1]{SMHCM} or \cite[Theorem~6.10]{SBGPRKT}. The next Theorem obtains a similar bound for seminormal simplicial affine semigroup rings.

\end{bem}

To get new bounds for the regularity of $K[B]$ we need a general bound for the regularity of a monomial ideal. The following is due to Hoa and Trung:

\begin{satz}[{\cite[Theorem~3.1]{SCMMI}}]\label{Shoa}

Let $I$ be a proper monomial ideal in $T$. Then
$$
\reg I \leq \var(I) - \height I + 1.
$$

\end{satz}

\begin{defi}

We define the set $\Gamma(B)\subseteq\{\Gamma_1,\ldots,\Gamma_f\}$ by $\Gamma_t\in\Gamma(B)$ for $t\in\{1,\ldots,f\}$ if $\reg K[B] = \reg I_t+\deg h_t$.

\end{defi}

By Equation~(\ref{Sregberechnung}) we obtain $\Gamma(B)\not=\emptyset$. Note that the ideals and shifts corresponding to the elements of $\Gamma(B)$ are computed by the function \texttt{regularityMA} in \cite{SBEN}.

\begin{prop}\label{Sd-1gamma}

Let $\Gamma_t\in\Gamma(B)$ for some $t\in\{1,\ldots,f\}$. If $\Gamma_t\subseteq \boxx(B)$, then
$$
\reg K[B]\leq \dim K[B]-1.
$$

\end{prop}
\begin{proof}

We need to show that $\reg I_t+\deg h_t\leq d-1$. In case that $\#\Gamma_t=1$ this follows from Remark~\ref{Sredd-1}. Assume that $\#\Gamma_t\geq2$; by Lemma~\ref{Svari} and Theorem~\ref{Shoa} we get
\begin{equation}\label{S1234}
\reg I_t\leq \var (I_t) - \height I_t + 1\leq d-1-\deg h_t - 2 + 1 = d-2-\deg h_t,
\end{equation}
since $\height I_t\geq2$. Hence $\reg I_t + \deg h_t\stackrel{\mbox{\scriptsize{(\ref{S1234})}}}{\leq} d-2$ and we are done.
\end{proof}

By Theorem~\ref{Scharsn} and Proposition~\ref{Sd-1gamma} we get the following theorem:

\begin{satz}\label{Sd-1}

If $B$ is seminormal, then
$$
\reg K[B]\leq \dim K[B]-1.
$$

\end{satz}

Note that the bound established in Theorem~\ref{Sd-1} is sharp. Assume $\alpha\geq d$ in Theorem~\ref{Sberechnung}; we get $\reg K[B_{d,\alpha}]=d-1$ and of course $B_{d,\alpha}$ is seminormal. Consider the monoid $B=\erz{(3,0,0),(0,3,0),(0,0,3),(2,1,0),(1,0,2),(0,2,1),(1,1,1)}$. One can show that $\Gamma_t=\{(2,2,2)\}$ for some $t$ and therefore $\Gamma_t\subseteq\boxx(B)$. Using \textsc{Macaulay2} \cite{SM2} we get $\reg K[B]=2$, hence $\Gamma_t\in\Gamma(B)$. Moreover, since $(4,2,0)\in B_A$ it follows that $K[B]$ is not seminormal by Theorem~\ref{Scharsn}. Thus, the condition in Proposition~\ref{Sd-1gamma} is not equivalent to $B$ being seminormal.

\begin{prop}\label{Sreggleichred}

Let $\Gamma_t\in\Gamma(B)$ for some $t\in\{1,\ldots,f\}$. If $\Gamma_t\subseteq\boxx(B)$ and \mbox{$\dim K[B]\leq5$}, then
$$
\reg K[B]=\red(K[B]).
$$

\end{prop}
\begin{proof}

We have $\red(K[B])\leq \reg K[B]$ by Equation~(\ref{Sregberechnung}). We show that $\reg I_t$ is equal to the maximal degree of a generator of $I_t$. By this we get
$$
\reg K[B]=\reg I_t+\deg h_t=\max\maxk{\deg x\mid x\in\Gamma_t},
$$
and hence $\red (K[B])\geq \reg K[B]$. Keep in mind that $I_t$ is squarefree. The case $\#\Gamma_t=1$ follows from construction. We therefore may assume that $\#\Gamma_t\geq2$, or equivalently, $\var(I_t)\geq2$; note that $\deg h_t\geq 1$, see Remark~\ref{Sgeq2}. Let $d\leq3$; by Lemma~\ref{Svari} we get $\var(I_t)\leq1$ which contradicts $\#\Gamma_t\geq2$. Let $d=5$; by Lemma~\ref{Svari} we have to consider the cases $\var(I_t)\in\{2,3\}$. Let $\var(I_t)=2$; the ideal $I_t$ is of the form $I_t=(y_k,y_l)T$ for some $k,l\in\{1,\ldots,5\}$ with $k\not=l$, since $\height I_t\geq2$. It follows that $\reg I_t=1$. By a similar argument we get the assertion for $d=4$ and $\var(I_t)=2$. Let $d=5$ and $\var(I_t)=3$. Since $\height I_t\geq2$ the only ideals possible are
$$
I_{t_1}=(y_k,y_l,y_m)T, I_{t_2}=(y_ky_l,y_m)T, I_{t_3}=(y_ky_l,y_ky_m,y_ly_m)T
$$
for some $k,l,m\in\{1,\ldots,5\}$ which are pairwise not equal. By Theorem~\ref{Shoa} we get $\reg I_{t_1}=1$ and $\reg I_{t_2}=\reg I_{t_3}=2$ and we are done.
\end{proof}

By Theorem~\ref{Scharsn} and Proposition~\ref{Sreggleichred} it follows that $\reg K[B]=\red (K[B])$ if $B$ is seminormal and $\dim K[B]\leq5$. Thus, the Eisenbud-Goto conjecture holds in this case by Proposition~\ref{Segred}. Theorem~\ref{Segsn} will confirm the conjecture in any dimension in the seminormal case. Note that Proposition~\ref{Sreggleichred} could fail for $d\geq6$. Let us consider the squarefree monomial ideal $I=(y_1y_2,y_3y_4)T$ with $\var (I)=4$. So $\reg I=3$ is bigger than the maximal degree of a generator of $I$ which is $2$.

\begin{lem}\label{Skey}

Let $\Gamma_t\subseteq \boxx(B)$ for some $t\in\{1,\ldots,f\}$. Let $n\in\Gamma_t$ and $m\in\tilde\Gamma_t$ such that $m=y^{(n-h_t)/\alpha}$. Then 

\begin{enumerate}

\item $n_{[q]}=0$ for all $q\in \supp (I_t)\setminus \supp(mT)$.

\item $n_{[q]}=\alpha$ for all $q\in \supp(mT)$.

\end{enumerate}

\end{lem}
\begin{proof}

$(1)$ Suppose to the contrary that there is a $q\in(\supp(I_t)\setminus \supp(mT))\not=\emptyset$ such that $n_{[q]}>0$. Since $q\in \supp(I_t)$ we have $h_{t[q]}=0$ by Lemma~\ref{Stechnisch} $(5)$, and therefore $n_{[q]}=\alpha$, since $h_{t[q]}-n_{[q]}\in\alpha\mathbb Z$ and $n_{[q]}\leq\alpha$. This implies $q\in \supp(mT)$ which is a contradiction.\\
$(2)$ Let $q\in \supp(mT)$; we have $n_{[q]}\geq\alpha$. Moreover, we get $n_{[q]}\leq\alpha$, since $\Gamma_t\subseteq\boxx(B)$.
\end{proof}

The above Lemma is false in general. For the affine semigroup $B$ in Example~\ref{Sbsprechnung} we have  $\Gamma_4=\{(6,2),(2,6)\}$, that is, $h_4=(2,2)$, and $\tilde\Gamma_4=\{y_1,y_2\}$. For $n\in\Gamma_4$ we get that $n_{[i]}>0$ for $i=1,2$. But $\supp(I_4)=\{1,2\}$ and $\#\supp(y_1T)=\#\supp(y_2T)=1$. As a consequence of the next proposition the Eisenbud-Goto conjecture holds if $B$ is seminormal.

\begin{prop}\label{Segsngamma}

Let $\Gamma_t\in\Gamma(B)$ for some $t\in\{1,\ldots,f\}$. If $\Gamma_t\subseteq \boxx(B)$, then
$$
\reg K[B]\leq \deg K[B]-\codim K[B].
$$

\end{prop}
\begin{proof}

By construction we need to show that $\reg I_t+\deg h_t\leq\deg K[B]-c$. If $\#\Gamma_t=1$ the assertion follows from Proposition~\ref{Segred}. Let $\#\Gamma_t\geq2$, equivalently, $I_t$ is a proper ideal, see Remark~\ref{Sgeq2}. We have $\Gamma_t=\{n_1,\ldots,n_{\#\Gamma_t}\}$ and $\tilde\Gamma_t=\{m_1,\ldots,m_{\#\Gamma_t}\}$; we may assume that $m_i=y^{(n_i-h_t)/\alpha}$. We set $J_k:=(m_1,\ldots,m_k)T$ and $g(k):=\var (J_k)-\height J_k+1+\deg h_t$ for $k\in\mathbb N$ with $1\leq k\leq \#\Gamma_t$. Note that $J_{\#\Gamma_t}=I_t$, moreover, $J_k$ is a (proper) squarefree monomial ideal in $T$, since $\Gamma_t\subseteq \boxx(B)$, hence $\var (J_k)=\#\supp(J_k)$. We show by induction on $k\in\mathbb N$ with $1\leq k\leq \#\Gamma_t$ that there is a set $L_k$ with the following properties
\begin{enumerate} 

\item[$(i)$] $L_k\subseteq B_A$.

\item[$(ii)$] $\#L_k\geq g(k)-1$.

\item[$(iii)$] $x\not\sim y$ for all $x,y\in L_k$ with $x\not=y$.

\item[$(iv)$] $\deg x\geq2$ for all $x\in L_k$.

\item[$(v)$] $x_{[q]}=0$ for all $x\in L_k$ and for all $q\in \supp(I_t)\setminus \supp(J_k)$.

\end{enumerate}
Let $k=1$. We have $\height J_1=1$ and $\var(J_1)+\deg h_t=\deg n_1$, that is, $g(1)= \deg n_1$. Fix a $\La\in\LA_{n_1}$ and set 
$$
L_1:=\{n_1(\La,0),\ldots,n_1(\La,\deg n_1-2)\},
$$ 
clearly $\#L_1= \deg n_1-1 = g(1)-1$, hence $(ii)$ is satisfied and by construction we get property $(iv)$. By Lemma~\ref{S*prop} $(1)$ $L_1\subseteq B_A$ which shows $(i)$, and by Lemma~\ref{S*prop} $(2)$ property $(iii)$ holds. By Lemma~\ref{Skey} $(1)$ we get $n_1(\La,0)_{[q]}=0$ for all $q\in \supp(I_t)\setminus \supp(J_1)$, hence $(v)$ holds by construction of $L_1$.

Using induction on $k\leq \#\Gamma_t-1$ the properties $(i)$-$(v)$ hold for $L_k$. We define the set $J:=\supp(m_{k+1}T)\setminus \supp(J_k)$. By Lemma~\ref{Skey} $(2)$ we get $n_{k+1[q]}=\alpha$ for all $q\in\supp(m_{k+1}T)$. Since $n_{k+1}\in B_A$ it follows that $\deg n_{k+1}\geq\#\supp(m_{k+1}T)+1$. Moreover, since $n_{k+1[q]}=\alpha$ for all $q\in J$ we can fix by Lemma~\ref{Shilfe2} a $\La\in\LA_{n_{k+1}}$ with the property: for all $p=1,\ldots,\#J$ there is a $q\in J$ with $0<n_{k+1}(\La,p)_{[q]}<\alpha$. There could be two different cases:\smallskip\\
\underline{Case 1:} $\supp(J_k)\cap \supp(m_{k+1}T)\not=\emptyset$.\quad(\mbox{e.\,g.}, $k=2, J_2=(y_1y_2,y_2y_3y_4)T$, and $m_3=y_4y_5y_6$.)\\
Set 
$$
L_{k+1}:=L_k\cup\{n_{k+1}(\La,1),\ldots,n_{k+1}(\La,\#J)\}.
$$
In case that $J=\emptyset$ we set $L_{k+1}:=L_k$.\\ 
$(iii)$ By induction we get $x\not\sim y$ for all $x,y\in L_k$ with $x\not=y$, moreover, $n_{k+1}(\La,i)\not\sim n_{k+1}(\La,j)$ for all $i,j\in\mathbb N$ with $0\leq i<j\leq\deg n_{k+1}$ by Lemma~\ref{S*prop} $(2)$. Fix an $x\in L_k$ and let $p\in\{1,\ldots,\#J\}$. By property $(v)$ $x_{[q]}=0$ for all $q\in J$, moreover, there is a $q\in J$ such that $0<n_{k+1}(\La,p)_{[q]}<\alpha$, hence $x\not\sim n_{k+1}(\La,p)$. Thus, property $(iii)$ is satisfied. This also shows that $\#L_{k+1}=\#L_k+\#J$.\\
$(i)$ By Lemma~\ref{S*prop} $(1)$ $n_{k+1}(\La,1),\ldots,n_{k+1}(\La,\#J)\in B_A$, since $n_{k+1}\in B_A$.\\
$(iv)$ Since $\#\supp(m_{k+1}T)\geq \#J+1$ we obtain $\deg n_{k+1}\geq\#J+2$. Hence $(iv)$ holds by construction.\\ 
$(v)$ By induction $x_{[q]}=0$ for all $x\in L_k$ and for all $q\in (\supp(I_t)\setminus \supp(J_k))\supseteq (\supp(I_t)\setminus \supp(J_{k+1}))$. By Lemma~\ref{Skey} $(1)$ we have $n_{k+1[q]}=0$ for all $q\in (\supp(I_t)\setminus \supp (m_{k+1}T))\supseteq (\supp(I_t)\setminus \supp(J_{k+1}))$, hence property $(v)$ holds by construction.\\ 
$(ii)$ Since $\supp(J_{k+1})=\supp(J_k)\cup\supp(m_{k+1}T)$ we get $\var(J_{k+1}) = \var(J_k)+\#J$. We have $\height J_{k+1}\geq\height J_k$ and therefore
$$
g(k+1) -1\leq \#J + \var(J_{k})-\height J_{k}+1+\deg h_t -1= \#J + g(k) - 1\leq \#J +\#L_k=\#L_{k+1}.
$$
\underline{Case 2:} $\supp (J_k)\cap \supp (m_{k+1}T)=\emptyset$.\quad(\mbox{e.\,g.}, $k=2, J_2=(y_1y_2,y_2y_3y_4)T$, and $m_3=y_5y_6y_7$.)\\
Note that $J=\supp(m_{k+1}T)$, in particular, $\#J\geq1$. Set
$$
L_{k+1}:=L_k\cup\{n_{k+1}(\La,1),\ldots,n_{k+1}(\La,\#J-1)\}.
$$
In case that $\#J=1$ we set $L_{k+1}:=L_k$.\\
$(iii), (i), (iv), (v)$ Analogous, replace $\#J$ by $\#J-1$ in the corresponding proofs in the first case. Moreover, $\#L_{k+1}=\#L_k+\#J-1$ by construction.\\
$(ii)$ We also have $\var(J_{k+1}) = \var(J_k)+\#J$. Since $\supp(J_k)\cap \supp(m_{k+1}T)=\emptyset$ we get that $m_{k+1}+J_k$ is a non-zero-divisor of $T/J_k$. Hence $\height J_{k+1}=\height J_k +1$ by Krull's Principal Ideal Theorem, see \cite[Theorem~10.1]{SEB}, and therefore
$$
g(k+1) -1= \#J + \var(J_{k})-\height J_{k}-1+1+\deg h_t -1= \#J + g(k) - 2\leq \#J+ \#L_k-1=\#L_{k+1}.
$$
By this we obtain a set $L_{\#\Gamma_t}$ with the above properties, in particular 
\begin{equation}\label{Skeyse}
\#L_{\#\Gamma_t}\stackrel{(ii)}{\geq} g(\#\Gamma_t)-1 = \var (I_t)-\height I_t+1+\deg h_t-1\geq\reg I_t+\deg h_t-1,
\end{equation}
by Theorem~\ref{Shoa}. We get a set
$$
L:=L_{\#\Gamma_t} \cup\{0,a_1,\ldots,a_c\},
$$
with $L\subseteq B_A$ such that $x\not\sim y$ for all $x,y\in L$ with $x\not= y$ by $(i)$, $(iii)$, and Remark~\ref{Sa1bisc}. Since $\deg K[B]=f$ (see Proposition~\ref{Szerlmi}) we have
$$
\deg K[B]\geq \#L\stackrel{(iv)}{=}\#L_{\#\Gamma_t}+c+1\stackrel{\mbox{\scriptsize{(\ref{Skeyse})}}}{\geq} \reg I_t+\deg h_t+c = \reg K[B]+c.
$$ 
\end{proof}

We therefore obtain from Theorem~\ref{Scharsn} and Proposition~\ref{Segsngamma} the following main result:

\begin{satz}\label{Segsn}

If $B$ is seminormal, then
$$
\reg K[B]\leq \deg K[B]-\codim K[B].
$$

\end{satz}

Note that the bound of Theorem~\ref{Segsn} is again sharp. For $d=2$ and $\alpha\geq2$ we get that $\reg K[B_{2,\alpha}]=\floor{2-\frac{2}{\alpha}}=1$ and $\deg K[B_{2,\alpha}]-\codim K[B_{2,\alpha}]=\alpha-(\alpha+1)+2=1$, see Section~\ref{Sfullveronese}.

\section{Regularity of full Veronese rings}\label{Sfullveronese}

For $X,Y\subseteq\mathbb N^d$ we define $X+Y:=\{x+y\mid x\in X, y\in Y\}$,  $nX:=X+\ldots+X$ ($n$-times), and $0X:=0$. Recall that $M_{d,\alpha}=\{(u_{1},\ldots,u_{d})\in\mathbb N^d\mid \sum_{i=1}^{d}u_{i}=\alpha\}$ and $B_{d,\alpha}$ denotes the submonoid of $(\mathbb N^d,+)$ which is generated by $M_{d,\alpha}$. For example $B_{2,2}=\erz{(2,0), (0,2), (1,1)}$. We have 
\begin{equation}\label{Schar}
n M_{d,\alpha} = \left\{(u_{1},\ldots,u_{d})\in\mathbb N^d\mid\sum\nolimits_{i=1}^{d}u_{i}=n \alpha\right\},
\end{equation}
hence there is an isomorphism of $K$-vector spaces: $K[B_{d,1}]_{n\alpha}\cong K[B_{d,\alpha}]_{n}$. It is a well known fact that \mbox{$h_{K[B_{d,1}]} (n)= \binom{n+d-1}{d-1}$}, where $h_M$ denotes the Hilbert polynomial. This shows that $h_{K[B_{d,\alpha}]}(n) = h_{K[B_{d,1}]} (n\alpha) =  \binom{n\alpha+d-1}{d-1}$ and therefore $\deg K[B_{d,\alpha}] = \alpha^{d-1}$. Moreover, we get $\codim K[B_{d,\alpha}] = \binom{\alpha+d-1}{d-1}-d$, since $\#M_{d,\alpha}= \binom{\alpha+d-1}{d-1}$. The semigroups $B_{d,\alpha}$ are normal, hence the ring $K[B_{d,\alpha}]$ is Cohen-Macaulay by \cite[Theorem~1]{SMHCM} and therefore $\#\Gamma_t=1$ for all $t=1,\ldots,f$, see \cite[Theorem~6.4]{SRSHFGA}. It follows that
\begin{equation}\label{Sgleichheit}
\reg K[B_{d,\alpha}] = \red(K[B_{d,\alpha}]),
\end{equation}
by Equation~(\ref{Sregberechnung}). In the following we will compute the reduction number $\red(K[B_{d,\alpha}])$ which can also be computed by $\red(K[B_{d,\alpha}])=\min\maxk{r\in\mathbb N \mid rM_{d,\alpha}+\{e_1,\ldots,e_d\}=(r+1)M_{d,\alpha}}$, see \cite[Page~129,135]{SHSCM}.

\begin{lem}\label{Sgro}

Let $r\in\mathbb N$. The following assertions are equivalent:

\begin{enumerate}

\item $rM_{d,\alpha} + \{e_1,\ldots,e_d\} = (r+1) M_{d,\alpha}$.

\item $(r+1)\alpha > d(\alpha-1)$.

\end{enumerate}

\end{lem}
\begin{proof}

$(1) \Rightarrow (2)$ Assume that $0\leq (r+1)\alpha \leq d(\alpha-1)$. There is an element $x\in\mathbb N^d$ with $x_{[j]}\leq\alpha-1$ for all $j=1,\ldots,d$ and $\sum_{j=1}^{d}x_{[j]}=(r+1)\alpha$. We have $x\in (r+1)M_{d,\alpha}$ by Equation~(\ref{Schar}). Suppose that $x \in rM_{d,\alpha} + \{e_1,\ldots,e_d\}$ we get $x = x' + e_j$ for some $x'\in\mathbb N^d$ and some $j\in\{1,\ldots,d\}$ which contradicts $x_{[j]}\leq\alpha-1$. Hence $x \notin rM_{d,\alpha} + \{e_1,\ldots,e_d\}$.\\
$(2) \Rightarrow (1)$ Let $x\in (r+1)M_{d,\alpha}$. Suppose that $x_{[j]}\leq\alpha-1$ for all $j=1,\ldots,d$. We get $(r+1)\alpha = \sum_{j=1}^{d} x_{[j]} \leq d(\alpha-1)$. Thus, $x_{[j]}\geq\alpha$ for some $j\in\{1,\ldots,d\}$ and therefore $x-e_j\in rM_{d,\alpha}$ by Equation~(\ref{Schar}). Hence $(r+1)M_{d,\alpha} \subseteq rM_{d,\alpha} + \{e_1,\ldots,e_d\}$, that is, $(r+1)M_{d,\alpha} = rM_{d,\alpha} + \{e_1,\ldots,e_d\}$ and we are done.
\end{proof}

\begin{satz}\label{Sberechnung}

We have
$$
\mbox{$\reg K[B_{d,\alpha}] = \floor{d - \frac{d}{\alpha}}$}.
$$

\end{satz}
\begin{proof}

By Equation~(\ref{Sgleichheit}) we need to show that $\red(K[B_{d,\alpha}]) = \floor{d - \frac{d}{\alpha}}$. We have 
$$
\mbox{$\left(\floor{d - \frac{d}{\alpha}}+1\right) \alpha > (d - \frac{d}{\alpha}) \alpha = d(\alpha-1)$},
$$
hence $\red(K[B_{d,\alpha}]) \leq \floor{d - \frac{d}{\alpha}}$ by Lemma~\ref{Sgro}. We may assume that \mbox{$\floor{d - \frac{d}{\alpha}}\geq1$}. We get
$$
\mbox{$\left(\floor{d - \frac{d}{\alpha}}-1+1\right) \alpha \leq \left(d - \frac{d}{\alpha}\right)\alpha = d (\alpha-1)$},
$$
hence $\red(K[B_{d,\alpha}]) > \floor{d - \frac{d}{\alpha}} -1$ by Lemma~\ref{Sgro} and we are done.
\end{proof}

\begin{bsp}\label{Sbadeg}

By Theorem~\ref{Sberechnung} we are able to compute the Castelnuovo-Mumford regularity of full Veronese rings. For $B_{20,2}$ we know that $\reg K[B_{20,2}] = \floor{20 - \frac{20}{2}} = 10$. Moreover, we have $\deg K[B_{20,2}]-\codim K[B_{20,2}] = 2^{19} - \binom{2+19}{19}+20 = 524098$.

\end{bsp}

\section*{Acknowledgement}

The author would like to thank J\"urgen St\"uckrad and L$\hat{\mbox e}$ Tu$\hat{\mbox a}$n Hoa for many helpful discussions.

\end{document}